\documentclass[11pt,reqno,twoside]{amsart}

\RequirePackage{geometry}

\begin{document}


\newcommand{\refe}[1]{(\ref{#1})}
\newcommand{\bq}{\begin{equation}}
\newcommand{\eq}{\end{equation}}
\newcommand{\ba}{\begin{array}}
\newcommand{\ea}{\end{array}}
\newcommand{\ZZ}{{\mathbb Z} }
\newcommand{\NN}{{\mathbb N}}
\newcommand{\CC}{{\mathbb C}}
\newcommand{\RR}{{\mathbb R}}
\newcommand{\btd}{\nabla}
\newcommand{\btu}{\Delta}
\newcommand{\dst}{\displaystyle}
\newcommand{\tsig}{\widetilde{\sigma}}
\newcommand{\ttau}{\widetilde{\tau}}


\theoremstyle{plain}
\newtheorem{teo}{Theorem}[section]
\newtheorem{cor}[teo]{Corollary}
\newtheorem{lema}[teo]{Lemma}
\newtheorem{exa}[teo]{Example}
\newtheorem{prop}[teo]{Proposition}
\newtheorem{rem}[teo]{Remark}


\title[Special Functions on the linear-type lattices]{On the Properties of Special
Functions on the linear-type lattices\footnote{May 15, 2011}}
\author{R. \'{A}lvarez-Nodarse}
\address{IMUS \& Departamento de An\'alisis Matem\'atico, Facultad de
Matem\'atica, Universidad de Sevilla. Apdo. Postal 1160, Sevilla,
E-41080, Sevilla, Spain}
\email{address ran@us.es}
\author{J. L. Cardoso}
\address{CM-UTAD \& Departamento de Matem\'atica, Universidade de
Tr\'as-os-Montes e Alto Dou\-ro. Apartado 202, 5001 - 911 Vila Real,
Portugal}
\email{jluis@utad.pt}

\begin{abstract}
We present a general theory for studying the difference analogues
of special functions of hypergeometric type on the linear-type
lattices, i.e., the solutions of the second order linear
difference equation of hypergeometric type on a special kind of
lattices: the linear type lattices. In particular, using the
integral representation of the solutions we obtain several
difference-recurrence relations for such functions. Finally, applications
to $q$-classical polynomials are given.
\end{abstract}
\maketitle
\subjclass{2000 Mathematics Subject Classification }{33D15, 33D45}

\section{Introduction}

The study of the so-called $q$-special functions has known an
increasing interest in the last years due its connection with several
problems in mathematics and mathematical-physics
(see e.g. \cite{ran04,aar,gr,nsu,vk}). A systematic study starting from the
second order linear difference equation that such
functions satisfy was started by Nikiforov and Uvarov in 1983
and further developed by Atakishiyev and Suslov (for a very nice reviews
see e.g. \cite{ARS,nsu,suslov}). Of particular interest is the so-called
$q$-classical polynomials (see e.g. \cite{ran-med}) introduced by
Hahn in 1949 which are polynomials on the lattice $q^s$.

Our main aim in this paper is to present a constructive approach
for generating recurrence relations and ladder-type operators for
the difference analogues of special functions of hypergeometric type
on the linear-type lattices. Here we will focus our attention on functions defined
on the $q$-linear lattice (for the linear lattice $x(s)=s$ see \cite{ran-luis}
and references therein, and for the continuous case see e.g. \cite{YDN:94}).
Therefore we will complete the work started in \cite{suslov} where
few recurrence relations where obtained. In fact we will prove, by using
the $q$-analoge of the technique introduced in \cite{ran-luis} for the
discrete case (uniform lattice), that the solutions (not only the polynomial ones)
of the difference equation on the $q$-linear lattice $x(s)=c_1q^s+c_2$
satisfy a very general recurrent-difference relation from where
several well known relations (such as the three-term recurrence relation
and the ladder-type relations) follow.

The structure of the paper is as follows: In section 2 the needed
results and notations from the $q$-special function theory are
introduced. In sections 3 and 4 the general theorems for obtaining recurrences
relations are presented. In section 5 the special case of classical
$q$-polynomials are considered in details and some examples
are worked out in details.

\section{Some preliminar results}

Here we collect the basic background \cite{ran,nsu,suslov} on
$q$-hypergeometric functions needed  in the rest of the work.

The hypergeometric functions on the non-uniform lattice $x(s)$
are the solutions of the second order linear difference
equation of hypergeometric type on non-uniform lattices
\begin{equation}
\begin{array}{c}
\dst \sigma(s) \frac{\Delta}{\Delta x(s-\frac{1}{2})}
\left[\frac{\nabla y(s)}{\nabla x(s)}\right] + \tau(s)
\frac{\Delta y(s)}{\Delta x(s)} + \lambda y(s) =0, \\[5mm]
\sigma(s)=\widetilde{\sigma}(x(s)) - \frac{1}{2}\widetilde{\tau}(x(s))
\Delta x\left(\mbox{$s-\frac{1}{2}$}\right), \quad
\tau(s)=\widetilde{\tau}(x(s)),
\end{array}
\label{difeq-q}
\end{equation}
where $\Delta y(s):= y(s+1)-y(s)$, $\nabla y(s):= y(s)-y(s-1)$,
are the forward and backward difference operators, respectively;
$\tsig(x(s))$ and $\ttau(x(s))$ are polynomials in $x(s)$ of
degree at most 2 and 1, respectively, and $\lambda$ is a constant.
Here we will deal with the linear and $q$-linear lattices, i.e.,
lattices of the form
\bq\label{red}
x(s)=c_1 s+c_2\quad\textrm{or}\quad x(s)=c_1(q)q^{s}+c_2(q),
\eq
respectively, with $\,c_1\neq 0\,$ and $\,c_1(q)\neq 0\,$.

We will define the {\em k-order difference derivative} of a solution
$y(s)$ of \refe{difeq-q} by
$$
y^{(k)}(s) := \btu^{(k)} [  y(s) ]=\frac{\btu}{\btu x_{k-1}(s)}
\frac{\btu}{\btu x_{k-2}(s)} \dots
\frac{\btu}{\btu x(s)} [ y(s) ] ,
$$
where $x_\nu(s)=x(s+\mbox{\footnotesize $\frac{\nu}{2}$})$. It is known \cite{nsu}
that $y^{(k)}(s) $ also satisfy a difference equation of the same type.
Moreover, for the solutions of the difference equation \refe{difeq-q}
the following theorem holds  
\begin{teo}{\cite{nsu1986,suslov}} The difference equation
\refe{difeq-q} has a particular solution of the form
\bq
\label{sol-d}
y_\nu(z)=\frac{C_\nu}{\rho(z)}\sum_{s=a}^{b-1}
\frac{\rho_\nu(s)\nabla x_{\nu+1}(s)}{[x_\nu(s)-x_\nu(z)]^{(\nu+1)}},
\eq
if the condition
$$
\left.\frac{\sigma(s)\rho_\nu(s)\nabla x_{\nu+1}(s)}
{[x_{\nu-1}(s)-x_{\nu-1}(z+1)]^{(\nu+1)}}\right|_a^b=0,
$$
is satisfied, and of the form
\bq
\label{sol-c}
y_\nu(z)=\frac{C_\nu}{\rho(z)}\int_{C}
\frac{\rho_\nu(s)\nabla x_{\nu+1}(s)}{[x_\nu(s)-x_\nu(z)]^{(\nu+1)}}ds,
\eq
if the condition
\bq
\label{con-c}
\int_C\Delta_s\frac{\sigma(s)\rho_\nu(s)\nabla x_{\nu+1}(s)}
{[x_{\nu-1}(s)-x_{\nu-1}(z+1)]^{(\nu+1)}}=0,
\eq
is satisfied. Here $C$ is a contour in the complex plane,
$C_\nu$ is a constant, $\rho(s)$ and $\rho_\nu(s)$ are the
solution of the Pearson-type equations
\bq\begin{split}
\frac{\rho(s+1)}{\rho(s)}=&\frac{\sigma(s)+\tau(s)\Delta
x(s-\frac{1}{2})}{\sigma(s+1)}=\frac{\phi(s)}{\sigma(s+1)},\\
\frac{\rho_\nu(s+1)}{\rho_\nu(s)}=&\frac{\sigma(s)+\tau_\nu(s)\Delta
x_\nu(s-\frac{1}{2})}{\sigma(s+1)}=\frac{\phi_\nu(s)}{\sigma(s+1)},
\end{split}
\label{pearson}
\eq
where
\bq
\tau_\nu(s)=\frac{\sigma(s+\nu)-\sigma(s)+\tau(s+\nu)\btu x(s+\nu-\frac{1}{2})}
{\btu x_{\nu-1}(s)},
\label{tau_n}
\eq
$\nu$ is the root of the equation
\bq \lambda_\nu +[\nu]_q \left\{
 \alpha_q(\nu-1)  \widetilde{\tau}' +
[\nu-1]_q \frac{\widetilde{\sigma}''}2  \right\}=0,
\label{lambda-q}
\eq
and $[\nu]_q$ and $\alpha_q(\nu)$ are the $q$-numbers
\bq\label{q-num}
[ \nu ]_q = \frac{q^{\nu/2} - q^{-\nu/2}}{q^{1/2} - q^{-1/2}},\qquad
\alpha_q(\nu) = \frac{q^{\nu/2} + q^{-\nu/2}}2,\quad
\forall\,\nu\in\CC,
\eq
respectively. The generalized powers $[x_{k}(s)-x_{k}(z)]^{(\nu)}$ are defined by
\bq\label{gen-pow-q}
[x_{k}(s)-x_{k}(z)]^{(\nu)}=(q-1)^{\nu} c_1^{\nu} q^{\nu(k-\nu+1)/2}q^{\nu z}
\frac{\Gamma_q(s-z+\nu)}{\Gamma_q(s-z)},\quad \nu\in\RR,
\eq
for the $\,q$-linear (exponential) lattice $x(s)=c_1q^s+c_2$ and
$$
[x_{k}(s)-x_{k}(z)]^{(\nu)}= c_1^{\nu}
\frac{\Gamma(s-z+\mu)}{\Gamma(s-z)},\quad \nu\in\RR,
$$
for the linear lattice $x(s)=c_1 s+c_2$, respectively.
For the definitions of the Gamma and the $\,q$-Gamma functions see, for
instance, \cite{aar}.
\end{teo}

\begin{rem} For the special case when $\nu\in\NN$, the
 generalized powers become
\[\begin{split}
[x_{k}(s)-x_{k}(z)]^{(n)}=&(-1)^{n} c_1^{n} q^{-n(n-1)/2}q^{n(z+k/2)}(q^{s-z};q)_n,\\
[x_{k}(s)-x_{k}(z)]^{(n)}=&c_1^n(s-z)_n,
\end{split}
\]
for $q$-linear and linear lattices, respectively.
\end{rem}
We will need the following straightforward proposition which proof we omit here
(see e.g. \cite{ran,suslov})
\begin{prop}\label{prop1}
Let $\,\mu\,$ and $\,\nu\,$ be complex numbers and $\,m\,$ and $\,k\,$ be
positive integers with $\,m\geq k\,.$ For the q-linear lattice $\:x(s)=c_1q^s+c_2\:$
we have
\begin{enumerate}
\item $\displaystyle\frac{\left[x_{\mu}(s)-x_{\mu}(z)\right]^{(m)}}
{\left[x_{\nu}(s)-x_{\nu}(z)\right]^{(m)}}
=q^{\frac{m(\mu-\nu)}{2}}\,,$
\item $\displaystyle\frac{\left[x_{\mu}(s)\!-\!x_{\mu}(z)\right]^{(m)}}
{\left[x_{\mu}(s)-x_{\mu}(z)\right]^{(k)}}
=\left[x_{\mu}(s)-x_{\mu}(z\!-\!k)\right]^{(m\!-\!k)}\,,$
\item $\displaystyle\frac{\left[x_{\mu}(s)-x_{\mu}(z)\right]^{(m)}}
{\left[x_{\nu}(s)-x_{\nu}(z)\right]^{(k)}}
=q^{\frac{k(\mu\!-\!\nu)}{2}}\left[x_{\mu}(s)\!-\!x_{\mu}
(z\!-\!k)\right]^{(m\!-\!k)}\,,$
\item $\dst\frac{\left[x_{\mu}(s)-x_{\mu}(z)\right]^{(m+1)}}
{\left[x_{\mu-1}(s+1)-x_{\mu-1}(z)\right]^{(m)}}=
x_{\mu-m}(s)-x_{\mu-m}(z)\,,$
\item $\dst\frac{\left[x_{\mu}(s)-x_{\mu}(z)\right]^{(m+1)}}
{\left[x_{\mu-1}(s)-x_{\mu-1}(z)\right]^{(m)}}=
x_{\mu-m}(s+m)-x_{\mu-m}(z)\,.$
\end{enumerate}
\end{prop}
To obtain the result for the linear lattice one only has to put in the
above formulas $q=1$.

\section{The general recurrence relation in the linear-type lattices}
In this section we will obtain several recurrence relations for the
solutions \refe{sol-d} and \refe{sol-c} of the difference equation
\refe{difeq-q} in the linear-type lattices \refe{red}. Since the equation
\refe{difeq-q} is linear we can restrict ourselves to the canonical cases
$x(s)=q^s$ and $x(s)=s$.

Let us define the functions\footnote{Obviously
the functions \refe{sol-d} correspond to the functions
\refe{phi-d}, whereas the functions $y_\nu$ given by \refe{sol-c}
correspond to those of \refe{phi-c}.}
\bq
\Phi_{\nu,\mu}(z)=\sum_{s=a}^{b-1}
\frac{\rho_\nu(s)\nabla x_{\nu+1}(s)}{[x_\nu(s)-x_\nu(z)]^{(\mu+1)}}
\label{phi-d}
\eq
and
\bq
\Phi_{\nu,\mu}(z)=\int_C
\frac{\rho_\nu(s)\nabla x_{\nu+1}(s)}{[x_\nu(s)-x_\nu(z)]^{(\nu+1)}}ds.
\label{phi-c}
\eq
Notice that the functions $y_\nu$ and the functions
$\Phi_{\nu,\mu}$ are related by the formula
\bq\label{con-y-phi}
y_\nu(z)=\frac{C_\nu}{\rho(z)}\Phi_{\nu,\nu}(z).
\eq

\begin{lema} For the functions $\Phi_{\nu,\mu}(z)$ the following relation holds
\bq
\nabla_z\:\Phi_{\nu,\mu}(z)=[\mu+1]_ q\nabla x_{\nu-\mu}(z)\Phi_{\nu,\mu+1}(z),
\label{2.4}
\eq
where $[t]_q$ denotes the symmetric $q$-numbers \refe{q-num}.
\end{lema}
\begin{proof} We will prove it for the functions \refe{phi-d}. The other case is analogous.
Using \refe{gen-pow-q}, one gets
\[
\begin{split}
\nabla_z\:&\Phi_{\nu,\mu}(z)  =  \displaystyle\sum_{s=a}^{b-1}
\nabla_z\left(\frac{\rho_{\nu}(s)\nabla
x_{\nu+1}(s)}{\left[x_{\nu}(s)-x_{\nu}(z)\right]^{(\mu+1)}}\right) \\
= & \displaystyle\sum_{s=a}^{b-1}\left(\frac{\rho_{\nu}(s)\nabla
x_{\nu+1}(s)}{\left[x_{\nu}(s)-
x_{\nu}(z)\right]^{(\mu+1)}}-\frac{\rho_{\nu}(s)\nabla x_{\nu+1}(s)}{\left[x_{\nu}(s)-
x_{\nu}(z-1)\right]^{(\mu+1)}}\right) \\
= & \displaystyle\sum_{s=a}^{b-1}\frac{\rho_{\nu}(s)\nabla x_{\nu+1}(s)}
{\left[x_{\nu}(s)-x_{\nu}(z-1)\right]^{(\mu)}}
\left(\frac{1}{x_{\nu}(s)-x_{\nu}(z)}-\frac{1}{x_{\nu}(s)-x_{\nu}(z-1-\mu)}\right) \\
 = & \displaystyle\sum_{s=a}^{b-1}\frac{\rho_{\nu}(s)\nabla x_{\nu+1}(s)}
 {\left[x_{\nu}(s)-x_{\nu}(z-1)\right]^{(\mu)}}
\frac{x_{\nu}(z)-x_{\nu}(z-1-\mu)}
{(x_{\nu}(s)-x_{\nu}(z))(x_{\nu}(s)-x_{\nu}(z-1-\mu))}  \\
= & \displaystyle\sum_{s=a}^{b-1}\frac{\rho_{\nu}(s)\nabla x_{\nu+1}(s)}
{\left[x_{\nu}(s)-x_{\nu}(z)\right]^{(\mu+2)}}
\left(x_{\nu}(z)-x_{\nu}(z-1-\mu)\right)
\end{split}
\]
Since $\:x(s)-x(s-t)=[t]_q\nabla x\left(s-\frac{t-1}{2}\right)$ we then have
\[
\begin{split}
\nabla_z\:\Phi_{\nu,\mu}(z)  = & \displaystyle\sum_{s=a}^{b-1}
\frac{\rho_{\nu}(s)\nabla x_{\nu+1}(s)}{\left[
x_{\nu}(s)-x_{\nu}(z)\right]^{(\mu+2)}}[\mu+1]_q\nabla
x_{\nu}\left(z-\frac{\mu}{2}\right) \\
= & [\mu+1]_q\nabla x_{\nu-\mu}(z)\Phi_{\nu,\mu+1}(z)
\end{split}
\]
which is (\ref{2.4}).
\end{proof}

{}From \refe{2.4} follows that
$$
\Delta_z\:\Phi_{\nu,\mu}(z)=[\mu+1]_ q\Delta x_{\nu-\mu}(z)\Phi_{\nu,\mu+1}(z+1).
$$
Next we prove the  following lemma that is the
discrete analog of the Lemma in \cite[page 14]{NU:88}.

\begin{lema}\label{lema1}.
Let $x(z)$ be $x(z)=q^z$ or $x(z)=z$. Then,
any three functions $\Phi_{\nu_i,\mu_i}(z)$, $i=1,2,3$, are connected
by a linear relation
\bq\label{rr-phi}
\sum_{i=1}^3 A_i(z)\Phi_{\nu_i,\mu_i}(z)=0,
\eq
with non-zero at the same time polynomial coefficients on $x(z)$,
$A_i(z)$, provided that the differences $\nu_i-\nu_j$ and
$\mu_i-\mu_j$, $i,j=1,2,3$, are integers and that the following
condition holds\footnote{In some cases this
condition is equivalent to the condition $x(s)^k\sigma(s)\rho_{\nu_0}(s)|_{s=a}^{s=b}=0$,
 $k=0,1,2,\ldots$. }
\bq
\label{con-phi-d}
\left.\frac{x^k(s)\sigma(s)\rho_{\nu_0}(s)}{[x_{\nu_0-1}(s)-x_{\nu_0-1}(z)]^{(\mu_0)}}
\right|_{s=a}^{s=b}=0,\quad k=0,1,2,\ldots ,
\eq
when the functions $\Phi_{\nu_i,\mu_i}$ are given by \refe{phi-d}
and
\bq
\label{con-phi-c}
\int_C \Delta_s \frac{x^k(s)\sigma(s)\rho_{\nu_0}(s)\, ds}
{[x_{\nu_0-1}(s)-x_{\nu_0-1}(z)]^{(\mu_0)}}=0,\quad
k=0,1,2,\ldots,
\eq
when $\Phi_{\nu_i,\mu_i}$ are given by \refe{phi-c}.
Here $\nu_0$ is the $\nu_i$, $i=1,2,3$, with the smallest
real part and  $\mu_0$ is the $\mu_i$, $i=1,2,3$, with the largest
real part.
\end{lema}
\begin{proof}
Since in \cite{ran-luis} we have proved the case when $x(s)=s$ (the uniform lattice)
we will restrict here to the case of the $q$-linear lattice $x(s)=c_1q^s+c_2$).
Moreover, we will give the proof for the case of functions of the form
\refe{phi-d}, the other case is completely equivalent. Using the
identity
$$\nabla
x_{\nu_i+1}(s)=q^{\frac{\nu_i-\nu_0}{2}}\nabla x_{\nu_0+1}(s),
$$
as well as (3) of Proposition \ref{prop1}, we have
\begin{equation*}
\begin{split}
\sum_{i=1}^3 & A_i(z)\Phi_{\nu_i,\mu_i}(z)=
\sum_{i=1}^3 A_i(z) \sum_{s=a}^{b-1}\frac{\rho_{\nu_i}(s)\nabla
x_{\nu_i+1}(s)}{\left[x_{\nu_i}(s)-x_{\nu_i}(z)\right]^{(\mu_i+1)}}\\
&=\sum_{s=a}^{b-1} \sum_{i=1}^3  A_i(z) \frac{\rho_{\nu_i}(s)\nabla
x_{\nu_i+1}(s)}{\left[x_{\nu_i}(s)-x_{\nu_i}(z)\right]^{(\mu_i+1)}}=
\sum_{s=a}^{b-1}\frac{1}{\left[x_{\nu_0}(s)-x_{\nu_0}(z)\right]^{(\mu_0+1)}}\times\\
& \left(
\dst\sum_{i=1}^3 A_i(z) q^{\frac{(\mu_i+1)(\nu_0-\nu_i)}{2}}
\left[x_{\nu_0}(s)\!-\!x_{\nu_0}(z\!-\!\mu_i\!-1)\right]^{(\mu_0-\mu_i)}\rho_{\nu_i}(s)
\nabla x_{\nu_i\!+\!1}(s)\right)\:\\
&=\sum_{s=a}^{b-1}\frac{\rho_{\nu_0}(s)\nabla
x_{\nu_0+1}(s)}{\left[x_{\nu_0}(s)-
x_{\nu_0}(z)\right]^{(\mu_0+1)}}\times \\
&\quad \left(\dst\sum_{i=1}^3 A_i(z)
q^{\frac{\mu_i(\nu_0-\nu_i)}{2}}
\left[x_{\nu_0}(s)-x_{\nu_0}(z-\mu_i-1)\right]^{(\mu_0-\mu_i)}
\frac{\rho_{\nu_i}(s)}{\rho_{\nu_0}(s)}\right).
\end{split}
\end{equation*}
Using the Pearson-type equation \refe{pearson} we obtain
\bq
\label{rho-vi-v0}
\rho_{\nu_i}(s)=\phi(s+\nu_0)\phi(s+\nu_0+1)\ldots\phi(s+\nu_i-1)\rho_{\nu_0}(s),
\eq
so
$$\sum_{i=1}^3 A_i(z)\Phi_{\nu_i,\mu_i}(z)=\sum_{s=a}^{b-1}\frac{\rho_{\nu_0}(s)\nabla
x_{\nu_0+1}(s)}{\left[x_{\nu_0}(s)-
x_{\nu_0}(z)\right]^{(\mu_0+1)}}\Pi(s)$$
where
\begin{equation}\label{Id-Pi}
\ba{c}\dst \Pi(s)=\sum_{i=1}^3
A_i(z)q^{\frac{\mu_i(\nu_0-\nu_i)}{2}}
\left[x_{\nu_0}(s)-x_{\nu_0}(z-\mu_i-1)\right]^{(\mu_0-\mu_i)}\times\\[3mm]
\qquad\qquad \phi(s+\nu_0)\phi(s+\nu_0+1)\cdots\phi(s+\nu_i-1)\:.
\ea
\end{equation}
Let us show that there exists a polynomial $\,Q(s)\,$ in $x(s)$
(in general, $\,Q\equiv Q(z,s)\,$ is a function of $z$ and $s$)
such that
\bq
\begin{split}
\frac{\rho_{\nu_0}(s)\nabla x_{\nu_0+1}(s)}{\left[x_{\nu_0}(s)-
x_{\nu_0}(z)\right]^{(\mu_0+1)}}
\Pi(s)=&\Delta\left[\frac{\rho_{\nu_0}(s-1)}{\left[x_{\nu_0-1}(s)-
x_{\nu_0-1}(z)\right]^{(\mu_0)}}
Q(s)\right]\\ = &
\Delta\left[\frac{\sigma(s)\rho_{\nu_0}(s)}{\left[x_{\nu_0-1}(s)-
x_{\nu_0-1}(z)\right]^{(\mu_0)}}
Q(s)\right]\:.
\end{split}
\label{3.8}
\eq
If such polynomial exists, then, taking the sum in $s$ from $s=a$ to $b-1$
and using the boundary conditions
\refe{con-phi-d} we obtain \refe{rr-phi}.

To prove the existence of the polynomial $\,Q(s)\,$ in the variable $\,x(s)\,$
in \refe{3.8} we write
\[\begin{split}
& \frac{\sigma(s+1)\rho_{\nu_0}(s+1)}{\left[x_{\nu_0-1}(s+1)-x_{\nu_0-1}(z)\right]^{(\mu_0)}}
Q(s+1)-\frac{\sigma(s)\rho_{\nu_0}(s)}{\left[x_{\nu_0-1}(s)-x_{\nu_0-1}(z)\right]^{(\mu_0)}}
Q(s)=\\
& \frac{\rho_{\nu_0}(s)}{\left[x_{\nu_0}(s)-x_{\nu_0}(z)\right]^{(\mu_0+1)}}\mbox{$\left[
\sigma(s+1)\frac{\rho_{\nu_0}(s+1)}{\rho_{\nu_0}(s)}\frac{\left[x_{\nu_0}(s)-
x_{\nu_0}(z)\right]^{(\mu_0+1)}}{\left[x_{\nu_0-1}(s+1)-x_{\nu_0-1}(z)\right]^{(\mu_0)}}
Q(s+1)-\right.$} \\
& \hspace{10em}\mbox{$\left.\sigma(s)\frac{\left[x_{\nu_0}(s)-
x_{\nu_0}(z)\right]^{(\mu_0+1)}}{\left[
x_{\nu_0-1}(s)-x_{\nu_0-1}(z)\right]^{(\mu_0)}}Q(s)\right]$}\:.
\end{split}
\]
{}From (4) and (5) of Proposition \ref{prop1},
and using (\ref{pearson}), the above expression becomes
\[\begin{array}{l}
\frac{\rho_{\nu_0}(s)}{\left[x_{\nu_0}(s)-x_{\nu_0}(z)\right]^{(\mu_0+1)}}\left\{
\phi_{\nu_0}(s)\left[x_{\nu_0-\mu_0}(s)-x_{\nu_0-\mu_0}(z)\right]Q(s+1)-\right. \\ [1em]
\hspace{10em}\left.\sigma(s)\left[x_{\nu_0-\mu_0}(s+\mu_0)-x_{\nu_0-\mu_0}(z)
\right]Q(s)\right\}\:.
\end{array}\]
Thus
\bq\label{eq-Pi-Q}\begin{array}{l}
(\sigma(s)+\tau_{\nu_0}(s)\nabla
x_{\nu_0+1}(s))\left[x_{\nu_0-\mu_0}(s)-x_{\nu_0-\mu_0}(z)\right]Q(s+1)- \\ [1em]
\hspace{2em}\sigma(s)\left[x_{\nu_0-\mu_0}(s+\mu_0)-x_{\nu_0-\mu_0}(z)\right]Q(s)
=\nabla x_{\nu_0+1}(s)\Pi(s).
\end{array}
\eq
Since $\,\nabla x_{\nu_0+1}(s)\,$ is a polynomial of degree one in $\,x(s)\,$,
$\,x_k(s)\,$ and $\tau_{\nu_0}(s)$ are polynomials of degree at
most one in $\,x(s)$, and $\sigma(s)$ is a polynomial of degree at most
two in $x(s)$, we conclude that the degree of $Q(s)$ is, at least, two
less than the degree of $\Pi(s)$, i.e., $\deg Q\geq \deg \Pi-2$.
Moreover, equating the coefficients of the powers of $\,x(s)=q^s\,$ on the
two sides of the above equation \refe{eq-Pi-Q}, we find a system of linear equations
in the coefficients of $Q(s)$ and the coefficients $\,A_i(z)\,$ which have
at least one unknown more then the number of equations. Notice that
the coefficients of the unknowns are polynomials in $\,q^z\,$, so that
after one coefficient is selected the remaining coefficients
are rational functions of $\,q^z\,$, therefore after multiplying
by the common denominator of the $\,A_i(z)\,$ we obtain the
linear relation with polynomial coefficients on $\,x\equiv x(z)=q^z\,$.
This completes the proof.
\end{proof}

The above Lemma when $q\to1$ and $x(s)=s$ leads to the corresponding
result on the uniform lattice $x(s)$ \cite{ran-luis}.

\subsection{Some representative examples}

In the following examples, and for the sake of simplicity, we will
use the notation
\begin{equation}\label{exp-stp}
\sigma(s)=aq^{2s}+bq^s+c,\, \tau(s)=dq^s+e, \,
\phi_\nu(s)=\sigma(s)+\tau_{\nu-1}(s)\nabla x_{\nu}(s)=fq^{2s}+gq^s+h.
\end{equation}

\begin{exa}\label{Ex3.4}\
The following relation holds
$$
A_1(z)\Phi_{\nu,\nu-1}(z)+A_2(z)\Phi_{\nu,\nu}+A_3(z)\Phi_{\nu+1,\nu}(z)=0,
$$
\noindent where the coefficients $\,A_1\,$, $\,A_2\,$ and
$\,A_3\,$, are polynomials in $\,x\equiv x(z)=q^z\,,$ given by
\begin{small}
\[
\begin{split}
A_1(z)=& -eq^{\frac{\nu}{2}}+\frac{b+e\left(q^{\frac{1}{2}}-q^{-\frac{1}{2}}\right)}{
a+d\left(q^{\frac{1}{2}}-q^{-\frac{1}{2}}\right)}
\left(dq^{\frac{\nu}{2}}+a[\nu]_q\right)+\Big(dq^{\nu}+a[2\nu]_q\Big)q^{\frac{\nu}{2}+z}, \\
A_2(z)=& \frac{c\left(dq^{\nu}\!+\!a[2\nu]_q\right)}{a\!+\!d\left(q^{\frac{1}{2}}\!-\!q^{-\frac{1}{2}}\right)}+
\frac{b\!+\!e\left(q^{\frac{1}{2}}\!-\!q^{-\frac{1}{2}}\right)}{q^{\frac{1}{2}}\!-\!q^{-\frac{1}{2}}}
\!\left(q^{\nu}\!+\frac{a}{q^{\nu}\left(a\!+\!d\left(q^{\frac{1}{2}}\!-\!q^{-\frac{1}{2}}\right)\right)}
\right)\!q^z+\Big(dq^{\nu}\!+\!a[2\nu]_q\Big)q^{2z}, \\
A_3(z)=& -\frac{dq^{\frac{\nu}{2}}+a[\nu]_q}{
a+d\left(q^{\frac{1}{2}}-q^{-\frac{1}{2}}\right)},
\end{split}
\]
\end{small}
where 
$a$, $b$, $c$, $d$, and $e$,
are the coefficients of $\sigma$ and $\tau$ \refe{exp-stp}.
\end{exa}

\medskip
\begin{proof}
Using the notations of Lemma \ref{lema1} we have $\,\nu_1=\nu\,$, $\,\nu_2=\nu\,$,
$\,\nu_3=\nu+1\,$, $\,\mu_1=\nu-1\,$, $\,\mu_2=\nu\,$ and
$\,\mu_3=\nu\,$, thus $\,\nu_0=\nu\,$ and $\,\mu_0=\nu\,.$
By \refe{Id-Pi}
\begin{small}
\begin{equation}\label{Ex1-1}
\!\Pi(s)=A_1\!\left(q^{s+\frac{\nu}{2}}\!-\!q^{z-\frac{\nu}{2}}\right)+A_2+
A_3q^{-\frac{\nu}{2}}\!
\left[\left(a\!+\!d\left(q^{\frac{1}{2}}\!-\!q^{-\frac{1}{2}}\right)\right)\!q^{2\nu+2s}+
\left(b\!+\!e\left(q^{\frac{1}{2}}\!-\!q^{-\frac{1}{2}}\right)\right)\!q^{\nu+s}+c\right].
\end{equation}
\end{small}
On the other hand, from (\ref{eq-Pi-Q}) and because $\,Q(s)=k\,$
is a constant --notice that $\,\deg{(\Pi)}=2$-- we have
\begin{small}
\begin{equation}\label{Ex1-2}
\begin{array}{l}
\!\!\nabla
x_{\nu_0+1}(s)\Pi(s)=k\left\{\left[\left(a+d\left(q^{\frac{1}{2}}-q^{-\frac{1}{2}}\right)
\right)q^{2\nu+2s}+\left(b+e\left(q^{\frac{1}{2}}-q^{-\frac{1}{2}}\right)
\right)q^{\nu+s}+c\right]\left(q^{s}-q^{z}\right)-\right. \\ [1em]
\hspace{6em}\left.\left(aq^{2s}+bq^s+c\right)\left(q^{\nu+s}-q^z\right)\right\}
\end{array}
\end{equation}
\end{small}
where $\,k\,$ is an arbitrary constant.
\noindent Introducing (\ref{Ex1-1}) in (\ref{Ex1-2}), using the
identity
$$\,\nabla x_{\nu_0+1}(s)=q^{\frac{\nu}{2}}\left(q^{\frac{1}{2}}-
q^{-\frac{1}{2}}\right)q^s\,$$ and comparing the coefficients of
the powers of $\,x(s)=q^s\,$ we get a linear system of three
equations with four variables $\,A_1\,$, $\,A_2\,$, $\,A_3\,$ and
$\,k\,$. Choosing $\,k=1\,$ and solving the corresponding system
we get, after some simplifications, the coefficients $\,A_1\,$,
$\,A_2\,$ and $\,A_3.$ 
\end{proof}

In the next examples, since the technique is similar to the previous one
we will omit the details.

\begin{exa}
The following relation holds
$$
A_1(z)\Phi_{\nu,\nu}(z)+A_2(z)\Phi_{\nu,\nu+1}(z)+A_3(z)\Phi_{\nu+1,\nu+1}(z)=0\,,
$$
\noindent where the coefficients $\,A_1\,$, $\,A_2\,$ and
$\,A_3\,$, are polynomials in $\,x\equiv x(z)=q^z\,,$ given by
\[
\begin{array}{l}
A_1(z)=f\,\big(a-f\,{q}^{2\,\nu}\big)q^{z}+a\,g\,q-f\,b\,{q}^{\nu+1} \,, \\ [0.7em]
A_2(z)={q}^{-\frac{\nu}{2}-1}\,\left(a-f\,{q}^{2\,\nu}\right)\,\left(f\,{q}^{2\,z}+g\,{q}^{z+1}+h\,{q}^{2}\right)\,, \\ [0.7em]
A_3(z)=\sqrt{q}\,\left(a\,q-f\,{q}^{\nu}\right),
\end{array}
\]
where $a$, $b$, $c$, $f$, $g$ and $h$,
are the coefficients of $\sigma$ and $\phi_\nu$ \refe{exp-stp}.
\end{exa}

\begin{exa}\label{Ex3.9}\
The following relation holds
$$
A_1(z)\Phi_{\nu-1,\nu-1}(z)+A_2(z)\Phi_{\nu,\nu-1}(z)+A_3(z)\Phi_{\nu,\nu}(z)=0\,,
$$
\noindent where the coefficients $\,A_1\,$, $\,A_2\,$ and
$\,A_3\,$, are polynomials in $\,x\equiv x(z)=q^z\,,$ given by
\begin{small}
\[
\begin{split}
\!\!\!A_1(z)= & {q}^{-\frac{1}{2}-\nu}\Big\{\,f{q}^{2z}\Big[-{a}^{2}h{q}^{4}+agb{q}^{\nu+4}-
{q}^{2\nu+2}\left( a{g}^{2}q-2fah+f{b}^{2}\right) \\ &
\hspace{2em}-fgb{q}^{3\nu+1}\left({q}^{2}-q-1\right)+
f{q}^{4\nu}\left({g}^{2}\left(q-1\right)q-fh\right)\Big]+ \\ &
\hspace{0.3em}g{q}^{z+1}\Big[\!-{a}^{2}h{q}^{5}+a{q}^{\nu+2}\left(gb{q}^{3}\!+\!fh{q}^{2}\!-\!fh\right)-
{q}^{2\nu+2}\big(\left(fah+fgb+a{g}^{2}\right){q}^{2}- \\ &
\hspace{0.3em}f\!\left(2ah\!-\!{b}^{2}\!+\!gb\right)q\!-\!fah\big)\!+\!f{q}^{3\nu}\!\left({q}^{2}\!
\left({g}^{2}q\!-\!fh\!+\!gb\!-\!{g}^{2}\right)\!+\!fh\right)\!+\!{f}^{2}h{q}^{4\nu}\!
\left({q}^{2}\!\!-\!q\!-\!1\right)\!\Big] \\ &
\hspace{0.3em}-{a}^{2}{h}^{2}{q}^{6}+agh{q}^{\nu+5}\left(bq+gq-g\right)+fgh{q}^{3\nu+4}\left(gq+b-g\right)
-{f}^{2}\,{h}^{2}q^{4\nu+2} \\ &
\hspace{0.3em}-h{q}^{2\nu+3}\Big(a{g}^{2}{q}^{3}+fgb{q}^{2}+f{g}^{2}{q}^{2}-2fahq+
f{b}^{2}q-2f{g}^{2}q-fgb+f{g}^{2}\Big)\Big\}\,, \\
\!\!\!A_2(z)= & \left(q^{-\frac{\nu}{2}}\!-\!q^{\frac{\nu}{2}}\right)
\!\left(f{q}^{2z}\!+\!g{q}^{z+1}\!+\!h{q}^{2}\right)\!
\Big(fq^z\!\left(fq^{2\nu}\!\!-\!aq^2\right)\!+\!fq^{\nu+1}\!\left(gq\!+\!b\!-\!g\right)\!-\!agq^3\Big), \\
\!\!\!A_3(z)= & f\left(f{q}^{\nu}\!-\!aq\right)
\Big[\big(fq^{2z}\!+\!hq^2\big)\left(fq^{2\nu}\!-\!aq^2\right)+gq^{z+1}\Big(fq^{\nu}
\big({q}^{\nu}\!+\!q\!-\!1\big)-a{q}^{3}\Big)\Big],
\end{split}
\]
\end{small}
where $a$, $b$, $c$, $f$, $g$ and $h$,
are the coefficients of $\sigma$ and $\phi_\nu$ \refe{exp-stp}.
\end{exa}

\begin{exa}
The following relation holds
$$
A_1(z)\Phi_{\nu-1,\nu-1}(z)+A_2(z)\Phi_{\nu,\nu}(z)+A_3(z)\Phi_{\nu,\nu+1}(z)=0\,,
$$
\noindent where the coefficients $\,A_1\,$, $\,A_2\,$ and
$\,A_3\,$, are polynomials in $\,x\equiv x(z)=q^z\,,$ given by
\begin{small}
\[
 \begin{array}{l}
\!\!A_1(z)={a}^{\!2}h{q}^{4}\!-\!agb{q}^{\nu+3}\!+{q}^{2\nu+2}
\!\left(f{b}^{\!2}\!-\!2fah\!+\!a{g}^{2}\right)\!-\!fgb{q}^{3\nu+1}\!+\!{f}^{2}h{q}^{4\nu}, \\ [0.7em]
\!\!A_2(z)={q}^{-\frac{1}{2}}\left(f{q}^{\nu}-a{q}^{2}\right)
\left(f{q}^{z+2\nu}-a{q}^{z+2}+g{q}^{2\nu+1}-b{q}^{\nu+2}\right), \\ [0.7em]
\!\!A_3(z)=-{q}^{\frac{v-3}{2}}\left(f\,{q}^{2v}-a{q}^{2}\right)\left({q}^{v+1}-1\right)
\left(g{q}^{z+1}+f{q}^{2z}+h{q}^{2}\right),
\end{array}
\]
\end{small}
where $a$, $b$, $c$, $f$, $g$ and $h$,
are the coefficients of $\sigma$ and $\phi_\nu$ \refe{exp-stp}.
\end{exa}

\begin{exa}\label{Ex3.11}\
The relation
$$
A_1(z)\Phi_{\nu,\nu-1}(z)+A_2(z)\Phi_{\nu,\nu}(z)+A_3(z)\Phi_{\nu+1,\nu+1}(z)=0\,,
$$
\noindent is verified when the polynomial coefficients $\,A_1\,$,
$\,A_2\,$ and $\,A_3\,$, in the variable $\,x\equiv x(z)=q^z\,,$
are given by
\begin{small}
\[
\begin{array}{l}
\!A_1(z)={q}^{\frac{\nu+1}{2}}\Big(f{q}^{z+\nu}\left(f{q}^{2\nu}\!-\!g{q}^{\nu}\!+\!b\!-\!a\right)-
f\left(h\!-\!b\right){q}^{2\nu+1}-aq\left(g{q}^{\nu}\!-\!h\right)\!\Big)\,, \\ [0.7em]
\!A_2(z)={q}^{-z+\nu+\frac{1}{2}}\Big({q}^{z}\left(f{q}^{2\nu}-a\right)+{q}^{\nu}\left(g{q}^{\nu}-b\right)\Big)
\left(f{q}^{2z}+g{q}^{z+1}+h{q}^{2}\right)\,, \\ [0.7em]
\!A_3(z)={q}^{2z}\left(f{q}^{\nu}-aq\right)+
q^{z+\nu}\Big(q\left(g{q}^{\nu}-aq-b\right)-f{q}^{\nu}\left({q}^{\nu+1}-q-1\right)\Big)+ \\ [0.5em]
\hspace{2.2em}{q}^{\nu+1}\big(\left(h\!-\!b\right)q^{\nu+1}+g{q}^{\nu}-h\big),
\end{array}
\]
\end{small}
where $a$, $b$, $c$, $f$, $g$ and $h$,
are the coefficients of $\sigma$ and $\phi_\nu$ \refe{exp-stp}.
\end{exa}

\section{Recurrences involving the solutions $y_\nu$}

In \cite{suslov} the following relevant relation was established
\bq
\label{con-del-y-phi}
\Delta^{(k)} y_{\nu}(s)=\frac{C_\nu^{(k)}}{\rho_k(s)}
\Phi_{\nu\,,\,\nu-k}(s),
\eq
where
$$
C_\nu^{(k)}=C_\nu\prod_{m=0}^{k-1}\left[\alpha_q(\nu+m-1)
\widetilde\tau'+[\nu+m-1]_q\frac{\widetilde\sigma''}2 \right].
$$

This relation is valid for solutions of the form \refe{sol-d}
and \refe{sol-c} of the difference equation \refe{difeq-q}.

In the following, $y^{(k)}_n(s)$ denotes the $k$-th differences
$\Delta^{(k)} y_n(s)$.

\begin{teo} \label{teo-rec-y}
In the same conditions as in Lemma \ref{lema1},
any three functions $y^{(k_i)}_{\nu_i}(s)$, $i=1,2,3$, are connected
by a linear relation
\bq
\sum_{i=1}^3 B_i(s) y^{(k_i)}_{\nu_i}(s)=0,
\label{rr-y}
\eq
where the $B_i(s)$, $i=1,2,3$, are polynomials.
\end{teo}
\begin{proof}
{}From Lemma \ref{lema1} we know that there exists three polynomials
$A_i(s)$, $i=1,2,3$ such that
$$
\sum_{i=1}^3 A_i(s) \Phi_{\nu_i,\nu_i-k_i}(s)=0,
$$
then, using the relation \refe{con-del-y-phi}, we find
$$
\sum_{i=1}^3 A_i(s)(C_\nu^{(k)})^{-1} \rho_{k_i}(s)
 y^{(k_i)}_{\nu_i}(s)=0.
$$
Now, dividing the last expression by $\rho_{k_0}(s)$, where
$k_0=\min\{k_1,k_2,k_3\}$, and using \refe{rho-vi-v0}
we obtain
$$
\sum_{i=1}^3 B_i(s) y^{(k_i)}_{\nu_i}(s)=0,\quad
B_i(s)=A_i(s)(C_\nu^{(k)})^{-1}\phi(s+k_0)\cdots\phi(s+k_i-1),
$$
which completes the proof.
\end{proof}

\begin{cor}
In the same conditions as in Lemma \ref{lema1}, the
following three-term recurrence relation holds
$$
A_1(s) y_{\nu}(s)+A_2(s) y_{\nu+1}(s) +A_3(s) y_{\nu-1}(s)=0,
$$
with polynomial coefficients  $A_i(s)$, $i=1,2,3$.
\end{cor}
\begin{proof}
It is sufficient to put $k_1=k_2=k_3=0$, $\nu_1=\nu$, $\nu_2=\nu+1$
and $\nu_3=\nu-1$ in \refe{rr-y}.
\end{proof}

\begin{cor}
In the same conditions as in Lemma \ref{lema1}, the
following $\Delta$-ladder-type relation holds
\bq
B_1(s) y_{\nu}(s)+B_2(s) \frac{\Delta y_{\nu}(s)}{\Delta x(s)} +B_3(s) y_{\nu+m}(s)=0,
\qquad m\in\ZZ,
\label{rr-y-del}
\eq
with polynomial coefficients  $B_i(s)$, $i=1,2,3$.
\end{cor}
\begin{proof}
It is sufficient to put $k_1=k_3=0$, $k_2=1$, $\nu_1=\nu_2=\nu$
and $\nu_3=\nu+m$ in \refe{rr-y}.
\end{proof}

Notice that for the case $m=\pm1$ \refe{rr-y-del} becomes
\bq
B_1(s) y_{\nu}(s)+B_2(s) \frac{\Delta y_{\nu}(s)}{\Delta x(s)}+B_3(s) y_{\nu+1}(s)=0,
\label{rr-y-del-rai}
\eq
\bq
\widetilde B_1(s) y_{\nu}(s)+\widetilde B_2(s) \frac{\Delta y_{\nu}(s)}{\Delta x(s)}+
\widetilde B_3(s) y_{\nu-1}(s)=0,
\label{rr-y-del-low}
\eq
with polynomial coefficients  $B_i(s)$ and $\widetilde B_i(s)$,
$i=1,2,3$. The above relations are usually called raising and
lowering operators, respectively, for the functions $y_\nu$.\\

Let us now obtain a raising and lowering operators
for the functions $y_\nu$ but associated to the $\nabla/\nabla x(s)$
operators.

We start applying the operator $\nabla/\nabla x(s)$ to \refe{con-y-phi}
\[\begin{split}
\frac{\nabla}{\nabla x(s)}y_{\nu}(s)=& \frac{\nabla}{\nabla x(s)}
\left[\frac{C_{\nu}}{\rho(s)}\Phi_{\nu,\nu}(s)\right]\\ =& \frac{1}{\nabla
x(s)}\left[C_{\nu}\Phi_{\nu\nu}(s)\left(\frac{1}{\rho(s)}-\frac{1}{\rho(s-1)}\right)+
\frac{C_{\nu}}{\rho(s-1)}\nabla\Phi_{\nu\nu}(s)\right],
\end{split}
\]
or, equivalently,
$$
\frac{\nabla\Phi_{\nu\nu}}{\nabla x(s)}=
\frac{\rho(s-1)}{C_{\nu}}\frac{\nabla y_{\nu}(s)}{\nabla
x(s)}-\frac{\Phi_{\nu\nu}(s)}{\nabla
x(s)}\left[\frac{\rho(s-1)}{\rho(s)}-1\right].
$$
By Lemma \refe{lema1} with $\,\nu_1=\mu_1=\nu_2=\nu\,,$ $\,\mu_2=\nu+1\,$
and $\nu_3=\mu_3=\nu+m$, there exist polynomial coefficients on $x(s)$, $A_i(s)$,
$i=1,2,3$, such that
$$
A_1(s)\Phi_{\nu,\nu}(s)+A_2(s)\Phi_{\nu,\nu+1}(s)+A_3(s)\Phi_{\nu+m,\nu+m}(s)=0.
$$
{}From \refe{2.4}
$$\Phi_{\nu,\nu+1}(s)=\frac{1}{[\nu+1]_q}\frac{\nabla
\Phi_{\nu,\nu}}{\nabla x(z)}=\frac{1}{[\nu+1]_q}\frac{\nabla
\Phi_{\nu,\nu}}{\nabla x(z)}.
$$
Therefore
\[\begin{split}
A_1(s)\Phi_{\nu,\nu}+\frac{A_2(s)}{[\nu+1]_q}&
\left[\frac{\rho(s-1)}{C_{\nu}}\frac{\nabla
y_{\nu}}{\nabla x(s)}-\frac{\Phi_{\nu\nu}(s)}{\nabla x(s)}\left(\frac{\rho(s-1)}{\rho(s)}-1
\right)\right]\\ & \qquad\qquad +A_3\Phi_{\nu+m,\nu+m}=0.
\end{split}
\]
Using now the Pearson equation \refe{pearson} and dividing by $\,\rho(s)\,$ we get
\[\begin{split}
A_1(s)y_{\nu}(s)+\frac{A_2(q)}{[\nu+1]_q}&\left[
\frac{\sigma(s)}{\phi(s-1)}\frac{\nabla
y_{\nu}}{\nabla x(s)}-\frac{y_{\nu}(s)}{\nabla x(s)}\left(\frac{\sigma(s)}{\phi(s-1)}-1
\right)\right]\\ & \qquad\quad +A_3\frac{C_{\nu}}{C_{\nu+m}}y_{\nu+m}(s)=0\:.
\end{split}
\]
Multiplying both sides by $\,[\nu+1]_q\phi(s-1)\,$,
\[\begin{split}
&A_1(s)[\nu+1]_q\phi(s-1)y_{\nu}(s) + A_2(s)\sigma(s)
\frac{\nabla y_{\nu}}{\nabla x(s)}-\\&
A_2(s)\frac{\sigma(s)-\phi(s-1)}{\nabla x(s)}y_{\nu}(s)
 +[\mu+1]_q C_{\nu}C_{\nu+m}^{-1}A_3\phi(s-1)y_{\nu+m}(s)=0.
\end{split}
\]
Thus we have proven the following
\begin{teo}
In the same conditions as in Lemma \ref{lema1}, the
following $\nabla$-ladder-type relation holds
\bq
C_1(s) y_{\nu}(s)+C_2(s) \frac{\nabla y_{\nu}(s)}{\nabla x(s)} +C_3(s) y_{\nu+m}(s)=0,
\qquad m\in\ZZ,
\label{rr-y-nab}
\eq
with polynomial coefficients  $C_i(s)$, $i=1,2,3$.
\end{teo}

Notice that for the case $m=\pm1$ \refe{rr-y-nab} becomes
\bq
C_1(s) y_{\nu}(s)+C_2(s) \frac{\nabla y_{\nu}(s)}{\nabla x(s)} + C_3(s) y_{\nu+1}(s)=0,
\label{rr-y-nab-rai}
\eq
\bq
\widetilde C_1(s) y_{\nu}(s)+\widetilde C_2(s) \frac{\nabla y_{\nu}(s)}{\nabla x(s)} y_{\nu}(s)+
\widetilde C_3(s) y_{\nu-1}(s)=0,
\label{rr-y-nab-low}
\eq
with polynomial coefficients  $C_i(s)$ and $\widetilde C_i(s)$,
$i=1,2,3$. The above relation are usually called raising and
lowering operators, respectively, for the functions $y_n$.
Eq. \refe{rr-y-nab-rai} was firstly obtained in \cite[Eq. (3.4)]{suslov}.

To conclude this section let us point that from formula \refe{con-del-y-phi}
and the examples \ref{Ex3.4}, \ref{Ex3.9}, and \ref{Ex3.11} follow the
relations
\begin{equation}\label{dif-rel3}
 \begin{split}
 & B_1(s) y^{(1)}_{\nu}(s)+B_2(s) y_{\nu}(s)+B_3(s) y^{(1)}_{\nu+1}(s)=0,\\
 & B_1(s) y^{(1)}_{\nu}(s)+B_2(s) y_{\nu-1}(s)+B_3(s) y_{\nu}(s)=0,\\
 & B_1(s) y^{(1)}_{\nu}(s)+B_2(s) y_{\nu}(s)+B_3(s) y_{\nu+1}(s)=0,
\end{split}
\end{equation}
respectively, being the last two expressions the lowering  and
raising operators for the functions $y_\nu$. Moreover, combining
the explicit values of $A_1$, $A_2$ and $A_3$ with formula \refe{con-del-y-phi}, one can
obtain the explicit expressions for the coefficients
$B_1$, $B_2$ and $B_3$ in \refe{dif-rel3}.

\section{Applications to $q$-classical polynomials}
In this section we will apply the previous results to the
$q$-classical orthogonal polynomials \cite{alv06,koor,medem-ran-paco} in order to show
how the method works. We first notice that these polynomials
are instances of the functions $y_\nu$ on the lattice $x(s)=q^s$
defined in \refe{sol-c}. In fact we have \cite{nsu,suslov}
\bq
P_n(x(s))= \frac{[n]_q! B_n}{\rho(s)\,\, 2\pi i}
\int_{C} \frac{\rho_n(z) \nabla x_{n+1}(z)}{ \left[ x_n(z)-x_n(s) \right]^{(n+1)} }dz,
\label{rep-int-q}
\eq
where $B_n$ is a normalizing constant, $C$ is a closed contour surrounding the points
$x=s,s-1,\dots,s-n$ and it is assumed that $\rho_n(s)=\rho(s+n) \prod_{m=1}^{n}\sigma(s+m)$
and $\rho_n(s+1)$ are analytic inside $C$ ($\rho$ is the solution of the Pearson equation \refe{pearson}), i.e.,
the condition \refe{con-c} holds.

A detailed study of the $q$-classical polynomials, including several
characterization theorems, was done in \cite{alv06,ks,medem-ran-paco}.
In particular, a comparative analysis of the $q$-Hahn tableau with
the $q$-Askey tableau \cite{ks} and Nikiforov-Uvarov tableau
\cite{nu} was done in \cite{ran-med}.
In the following we use the standard notation for the $q$-calculus \cite{gr}.
In particular by $(a;q)_k= \prod_{m=0}^{k-1}(1-aq^m)$,
we denote the $q$-analogue of the Pochhammer symbol.

Since the $q$-classical polynomials are
defined by \refe{rep-int-q} where the contour $C$ is closed and $\nu$
is a non-negative integer, then the condition \refe{con-phi-c}
is automatically fulfilled, so Lemma \ref{lema1} holds for all of them.
Moreover, the Theorem \ref{teo-rec-y} holds and there exist
the non vanishing polynomials $B_1$, $B_2$ and $B_3$ of \refe{rr-y}.

In the following we will assume that the three term recurrence relation is
known, i.e.,
\begin{equation}\label{Eq4.1}
\begin{array}{l}
x(s) P_n(x(s))=\alpha_n P_{n+1}(x(s))+\beta_nP_n(x(s))+\gamma_nP_{n-1}(x(s))=0, \quad
\;\;n\geq0 \\ [1em]
P_{-1}(x(s))=0\,,\;\;P_0(x(s))=1,\qquad x(s)=q^s.
\end{array}
\end{equation}
where the coefficients $\,\alpha_n\,$, $\,\beta_n\,$ and $\,\gamma_n\,$
can be computed using the coefficients $\,\sigma\,$, $\,\tau\,$ and
$\,\lambda\equiv \lambda_n\,$ of \refe{difeq-q}, being
$\,\lambda_n\,$ given by \refe{lambda-q} and \refe{q-num} with $\,\nu=n\,$.
For more details see, e.g., \cite{ran,medem-ran-paco}.

Since the TTRR and the differentiation formulas for the
$q$-polynomials are very well known (see e.g.
\cite{ks,medem-ran-paco,suslov}) we will obtain here two
recurrent-difference relations involving the
$q$-differences of the polynomials and
the polynomials themselves.

\subsection{The first difference-recurrece relation}

If we choose $\,\nu_1=n-1\,$, $\,\nu_2=n\,$, $\,\nu_3=n+1\,$,
$\,k_1=1\,$, $\,k_2=1\,$ and $\,k_3=0$, in Theorem
\ref{teo-rec-y} one gets
\begin{equation*}
A_1(s)\Delta^{(1)}P_{n-1}(x(s))+A_2(s)\Delta^{(1)}P_{n}(x(s))+
A_3(s)P_{n+1}(x(s))=0\,.
\end{equation*}
Using \cite[Eq. (6.14), page 193]{ran}
\begin{equation*}
[ \sigma(s) + \tau(s) \btu x(s-1/2)]\Delta^{(1)}P_{n}(x(s))  =
\widehat{\alpha}_n  P_{n+1}(x(s))  +
\widehat{\beta}_n P_{n}(x(s))  + \widehat{\gamma}_n  P_{n-1}(x(s)) ,
\end{equation*}
where
\begin{equation*}
\ba{c}
 \hspace{-.3cm} \displaystyle \widehat{\alpha}_n=   \displaystyle \frac{ \lambda_n}{[n]_q }
 \left[ q^{ -\frac{n}{2} } \alpha_n - \frac{B_{n}}{\tau'_n B_{n+1}} \right],
  \hspace{.3cm}
 \widehat{\beta}_n=  \displaystyle \frac{ \lambda_n}{[n]_q}
 \left[ q^{ -\frac{n}{2} } \beta_n +
 \frac{\tau_{n}(0)}{\tau'_n } - c_3(q^{-\frac{n}2}-1)\right], \\
\\
\quad  \widehat{\gamma}_n=    \displaystyle \frac{\lambda_n q^{-\frac{n}{2} }\gamma_n}{[n]_q},
\ea
\end{equation*}
to compute $\,\Delta^{(1)}P_{n}(x(s))=\frac{\Delta P_n(x(s))}{\Delta x(s)}\,$ we get
$$
\begin{array}{l}
\left[A_2(s)\frac{\lambda_n}{[n]_q}\left(q^{-\frac{n}{2}}\alpha_n-
\frac{B_n}{\tau_n^{\prime}B_{n+1}}\right)+\left(\sigma(s)+
\tau(s)\Delta x\left(s-\frac{1}{2}\right)\right)A_3(s)\right]P_{n+1}+ \\ [1em]
\left[A_1(s)\frac{\lambda_{n-1}}{[n-1]_q}\left(q^{-\frac{n-1}{2}}\alpha_{n-1}-
\frac{B_{n-1}}{\tau_{n-1}^{\prime}B_{n}}\right)+A_2(s)\frac{\lambda_n}{[n]_q}\left(
q^{-\frac{n}{2}}\beta_n+\frac{\tau_n(0)}{\tau_n^{\prime}}\right)\right]P_{n}+ \\ [1em]
\left[A_1(s)\frac{\lambda_{n-1}}{[n-1]_q}\left(q^{-\frac{n-1}{2}}\beta_{n-1}+
\frac{\tau_{n-1}(0)}{\tau_{n-1}^{\prime}}\right)+
A_2(s)\frac{\lambda_{n}q^{-\frac n 2}\gamma_n}{[n]_q}\right]P_{n-1}+ \\ [1em]
A_1(s)\frac{\lambda_{n-1}q^{-\frac{n-1}{2}}\gamma_{n-1}}{[n-1]_q}P_{n-2}=0\,,
\end{array}
$$
By (\ref{Eq4.1}) we may write
$$
P_{n-2}(x(s))=\frac{x(s)-\beta_{n-1}}{\gamma_{n-1}}
P_{n-1}(x(s))-\frac{\alpha_{n-1}}{\gamma_{n-1}}P_n(x(s))
$$
so the above equality becomes
\begin{equation}\label{Eq4.3}
\begin{array}{c}
\left[\frac{\lambda_n}{[n]_q}\left(q^{-\frac{n}{2}}\alpha_n-
\frac{B_n}{\tau_n^{\prime}B_{n+1}}\right)A_2(s)+\left(\sigma(s)+
\tau(s)\Delta x\left(s-\frac{1}{2}\right)\right)A_3(s)\right]P_{n+1}(x(s))+ \\ [1em]
\left[-\frac{\lambda_{n-1}}{[n-1]_q}\frac{B_{n-1}}{\tau_{n-1}^{\prime}B_{n}}A_1(s)+
\frac{\lambda_n}{[n]_q}\left(
q^{-\frac{n}{2}}\beta_n+\frac{\tau_n(0)}{\tau_n^{\prime}}\right)A_2(s)\right]P_{n}(x(s))+ \\ [1em]
\left[\frac{\lambda_{n-1}}{[n-1]_q}\left(\frac{\tau_{n-1}(0)}{\tau_{n-1}^{\prime}}+
q^{-\frac{n-1}{2}}x\right)A_1(s)+
\frac{\lambda_{n}}{[n]_q}q^{-\frac n 2}\gamma_nA_2(s)\right]P_{n-1}(x(s))=0\,.
\end{array}
\end{equation}
Comparing the above equation with the TTRR \refe{Eq4.1} one can obtain the
explicit values of $A_1$, $A_2$, and $A_3$.

\subsubsection{Some examples}

Since we are working in the $q$-linear lattice $x(s)= q^s$,
for the sake of simplicity, we will use the letter $x$ to denote the variable of
the polynomials \cite{ks,medem-ran-paco}. We will consider monic polynomials, i.e.,
those with the leading coefficient equal to 1. In the following we need the value of
$\tau_n(x)$ for each family, which can be computed using (\ref{tau_n}).

\subsubsection*{Al-Salam-Carlitz I $q$-polynomials}

\noindent For the  Al-Salam-Carlitz I monic polynomials $U_{n}^{(a)}(x;q)$ we have
(see \cite[see table 6.5, p.208]{ran} or \cite{medem-ran-paco})
$$
\begin{array}{l}
\sigma(x)=(1-x)(a-x)\,, \quad\tau_n(x)=\frac{q^{\frac {1-n}{2}}}{1-q}\big(x-(1+a)\big)\,,
\\ [0.5em]
\tau(x)=\tau_0(x)\,, \quad \lambda_n=-\frac{q^{\frac{3}{2}-n}\left(1-q^n\right)}{(1-q)^2}\,,
\end{array}
$$
and
$$
\alpha_n=1\,, \quad \beta_n=(1+a)q^n\,,\quad \gamma_n=-aq^{n-1}\left(1-q^n\right)\,.
$$
The constant $\,B_n\,$ is given by \cite[Eq. (5.57), p. 147]{ran},
$B_n=q^{\frac{1}{4}n(3n-5)}(1-q)^n$.
Introducing these values into the equation (\ref{Eq4.3}) it becomes
\begin{equation*}
\begin{split}
\left[q\left(q^{-\frac n 2}-1\right)A_2(x)+a(1-q)q^nA_3(x)\right]U_{n+1}^{(a)}(x;q)+ & \\
\left[q^{-\frac{n}{2}-\frac{5}{2}}A_1(x)+
q^{1+\frac n 2}(1+a)\left(1-q^{\frac n 2}\right)A_2(x)\right]U_{n}^{(a)}(x;q)+ & \\
\left[\left(q^{\frac{n+3}{2}}(1+a)-q^{2-n}x\right)A_1(x)+
aq^{n}\left(1-q^{n}\right)A_2(x)\right]U_{n-1}^{(a)}(x;q)=&\, 0.
\end{split}
\end{equation*}
Comparing with the TTRR (\ref{Eq4.1}) for the Al-Salam I polynomials we obtain
a linear system for getting the unknown coefficients $\,A_1\,$, $\,A_2\,$ and $\,A_3\,$
\begin{equation*}
\begin{array}{l}
q\left(q^{-\frac n 2}-1\right)A_2(x)+a(1-q)q^nA_3(x)=1\,, \\ [1em]
q^{-\frac{n}{2}-\frac{5}{2}}A_1(x)+
q^{1+\frac n 2}(1+a)\left(1-q^{\frac n 2}\right)A_2(x)=(1+a)q^n-x\,, \\ [1em]
\left(q^{\frac{n+3}{2}}(1+a)-q^{2-n}x\right)A_1(x)+
aq^{n}\left(1-q^{n}\right)A_2(x)=aq^{n-1}\left(q^n-1\right)\,.
\end{array}
\end{equation*}
The solution of the above system is
\begin{equation}\label{alsalam}
\begin{array}{l}
A_1(x)=\frac{aq^n\left(1+q^{\frac n 2}\right)\left((1+a)-
q^{-\frac n 2}x\right)}{aq^{-\frac 5 2}\left(1+q^{\frac n 2}\right)-
q(1+a)\left(q^{\frac{n+3}{2}}(1+a)-q^{2-n}x\right)}\,, \\ [1.3em]
A_2(x)=\frac{-aq^{-\frac 7 2}\left(1-q^{n}\right)-\big((1+a)q^n-x\big)
\left(q^{\frac 3 2}(1+a)-q^{2-\frac{3n}{2}}x\right)}{ \left(1-q^{\frac n 2}\right)
\left[aq^{-\frac 5 2}\left(1+q^{\frac n 2}\right)-q(1+a)
\left(q^{\frac{n+3}{2}}(1+a)-q^{2-n}x\right)\right]}\,, \\ [1.3em]
A_3(x)=\frac{a+q^{\frac{11}{2}-2n}x^2+q^{-\frac n 2}\big(a-(1+a)q^5x\big)}{
a(1-q)\left[aq^n+q^\frac{3n}{2}\big(a-(1+a)^2q^5\big)+
(1+a)q^{\frac{11}{2}}x\right]}\,.
\end{array}
\end{equation}
Then, the Al-Salam I $q$-polynomials satisfy the the following relation
\begin{equation}\label{Eq4.4}
A_1(x)\Delta^{(1)}U_{n-1}^{(a)}(x;q)+A_2(x)\Delta^{(1)}U_{n}^{(a)}(x;q)+
A_3(x)U_{n+1}^{(a)}(x;q)=0\,,
\end{equation}
where the coefficients $\,A_1\,$, $\,A_2\,$ and $\,A_3\,$ are  given by \refe{alsalam}.

Notice that the coefficients $\,A_1\,$, $\,A_2\,$ and $\,A_3\,$
are rational functions on $x$. Therefore, multiplying \refe{Eq4.4}
by and appropriate factor it becomes a linear relation with
polynomials coefficients.

\subsubsection*{Alternative $q$-Charlier polynomials}

\noindent In this case (see \cite[table 6.6, p.209]{ran})
$$
\begin{array}{l}
\sigma(x)=q^{-1}x(1-x)\,, \quad \tau_n(x)=
-\frac{q^{-\frac{n+1}{2}}}{1-q}\Big(\left(1+aq^{1+2n}\right)x-1\Big), \\[0.5em]
\tau(x)=\tau_0(x)\,, \quad 
\lambda_n=\frac{q^{\frac{1}{2}-n}\left(1-q^n\right)\left(1+aq^{n}\right)}{(1-q)^2}\,,
\end{array}
$$
and, for the monic case, $\alpha_n=1$
$$
\begin{array}{l} \beta_n=\dst \frac{q^n\Big(1+aq^{n-1}+aq^n-aq^{2n}\Big)}{
\left(1+aq^{2n-1}\right)\left(1+aq^{2n+1}\right)} \,,\quad
\dst \gamma_n=\frac{aq^{3n-2}\left(1-q^n\right)\left(1+aq^{n-1}\right)}{
\left(1+aq^{2n-2}\right)\left(1+aq^{2n-1}\right)^2\left(1+aq^{2n}\right)} \,.
\end{array}
$$
The corresponding normalizing constant $\,B_n\,$ is given by
$$
B_n=\frac{(-1)^nq^{\frac{1}{4}n(3n-1)}(1-q)^n}{\left(-aq^{n};q\right)_n}\,.
$$
Following the same procedure as before we obtain the following relation
for the alternative Charlier $q$-polynomials:
\begin{equation*}
\begin{array}{c}
A_1(x)\Delta^{(1)}K_{n-1}(x;a;q)+A_2(x)\Delta^{(1)}K_{n}(x;a;q)+
A_3(x)K_{n+1}(x;a;q)=0\,,
\end{array}
\end{equation*}
with the coefficients
\begin{small}
\begin{equation*}
\begin{split}
\!\!A_1(x)=& \frac{a\left(1+aq^{\frac n 2}\right)\left(\left(1+aq^{2n+1}
\right)x-q^{-\frac{n}{2}}\right)x}{q^2\left(1+aq^{2n-2}\right)
\left(1+aq^{2n-1}\right)\left(1+aq^{2n}\right)\left(1+aq^{2n+1}\right)} \,,
\\
\!\!A_2(x)=& \frac{-q^{\frac{3n+1}{2}}\left(1\!+\!aq^{n}\right)x\!+\!
\left(1\!+\!aq^{2n}\right)\left(q^{1+\frac{n}{2}}\left(1\!+\!aq^{2n\!+\!1}\right)+
aq^{2n\!+\!\frac{1}{2}}(1\!+\!q)+q^{\frac{3n}{2}}\left(1\!-\!aq^{2n}\right)\right)x^2}{
q^{3n}\left(1+aq^{n}\right)\left(1+aq^{2n}\right)\left(1+aq^{2n+1}\right)}-
\\ &
\frac{q^{\frac{3}{2}}\left(1+aq^{2n-1}\right)\left(1+aq^{2n+1}\right)x^3}{
q^{3n}\left(1+aq^{n}\right)\left(1+aq^{2n}\right)\left(1+aq^{2n+1}\right)} \,,
\\
\!\!A_3(x)= & \frac{q^{\frac{n+1}{2}}+aq^{2n}\left(q^{\frac{n}{2}}+1+
q^{\frac{1}{2}}\right)-q^{\frac{3}{2}}\left(1-aq^{\frac{3n}{2}}\right)
\left(1+aq^{2n-1}\right)x}{q^{\frac{9n}{2}}\left(1+aq^{n}\right)} \,.
\end{split}
\end{equation*}
\end{small}

\subsubsection*{Big $q$-Jacobi polynomials}


\noindent In this case (see \cite[see table 6.2, p.204]{ran} or \cite{medem-ran-paco})
$$
\begin{array}{l}
\sigma(x)=q^{-1}(x-aq)(x-cq)\,,\dst
\lambda_n=-q^{\frac{1}{2}-n}\frac{\left(1-abq^{1+n}\right)
\left(1-q^{n}\right)}{(1-q)^2}\,, \\ [0.5em]
\tau_n(x)=\dst\frac{q^{\frac {1-n}{2}}}{1-q}\left(\frac{1-abq^{2+2n}}{q}x+
a(b+c)q^{1+n}-(a+c)\right)\,,
\tau(x)=\tau_0(x)\,,
\end{array}
$$
and, for the monic case $\alpha_n=1$,
\begin{equation*}
\begin{split}
\beta_n= & \mbox{$\frac{c+a^2bq^n\Big((1+b+c)q^{1+n}-q-1\Big)+a\Big(
1+b+c-q^n\big(b(1+q)+c\big(1+q+b+bq-bq^{1+n}\big)\big)\Big)}{
q^{-1-n}\left(1-abq^{2n}\right)\left(1-abq^{2n+2}\right)}$} \,,\\
\gamma_n=& -\frac{a\left(1-q^n\right)\left(1-aq^n\right)\left(1-bq^n\right)
\left(1-cq^n\right)\left(c-abq^n\right)}{q^{-1-n}\left(1-abq^{2n-1}\right)
\left(1-abq^{2n}\right)^2\left(1-abq^{2n+1}\right)} \,.
\end{split}
\end{equation*}
The corresponding normalizing constant  is
$$
B_n=\frac{(1-q)^nq^{\frac{1}{4}n(3n-1)}}{\left(abq^{1+n};q\right)_n}.
$$
The big $q$-Jacobi polynomials satisfy the following relation
\begin{equation*}
\begin{split}
A_1(x)\Delta^{(1)}p_{n-1}(x;a,b,c;q)+A_2(x)\Delta^{(1)}p_{n}(x;a,b,c;q)+ &
\\
A_3(x)p_{n+1}(x;a,b,c;q) & =0\,,
\end{split}
\end{equation*}
with the coefficients $\,A_1\,$, $\,A_2\,$ and $\,A_3\,$ given by
\begin{footnotesize}
\begin{equation*}
\begin{array}{l}
A_1(x)=\frac{aq^{-\frac{1}{2}+n}\left(1-abq^{n+1}\right)
(1-x)(c-bx)\Big(c-(b+c)x+bx^2\Big)}{1-abq^{2n-1}}\times \\ [0.7em]
\hspace{0.8em}  \Bigg\{(1-q)q^{\frac{n}{2}}\big(1-abq^{2n+2}\big)\!
\left[\frac{c+a\Big(1+b+c+b(c+a(1+b+c))q^{2n+1}-
(c+b(1+a+c))q^n(1+q)\Big)}{q^{-(n+1)}\left(1-abq^{2n}\right)
\left(1-abq^{2n+2}\right)}-x\right]\!D(x)- \\ [0.7em]
(1-q)q^n\left(1-abq^{2n}\right)\Big[\left(1-abq^{2n}\right)
\left(-c+a\big(-1+(b+c)q^{n+1}\big)\right)+ \\ [0.7em]
 q^{\frac{n}{2}}\left(c+a\Big(1+b+c+b\big(c+a(1+b+c)\big)q^{2n+1}-
\big(c+b(1+a+c)\big)q^{n}(1+q)\Big)\right)N(x)\Big]\Bigg\}, \\ [1.3em]
A_2(x)=a(1-q)q^n\left(1-abq^{2n}\right)^2
\left(1-abq^{2n+2}\right)(1-x)(c-bx)\Big(c-(b+c)x+bx^2\Big)N(x), \\ [1.3em]
A_3(x)=\left(1-abq^{n+1}\right)\left(1-abq^{2n+2}\right)(1-x)(c-bx)D(x)+ \\ [0.6em]
\hspace{3em}q^{-1-\frac{n}{2}}\Big(1-q^{\frac{n}{2}}\Big)\left(1+abq^{1+\frac{3n}{2}}\right)
\left(1-abq^{2n}\right)^2\left(1-abq^{2n+2}\right)\Big(c-(b+c)x+bx^2\Big)N(x),
\end{array}
\end{equation*}
\end{footnotesize}
where the polynomials $\,N(x)\,$ and $\,D(x)\,$ are given by
\begin{small}
\[
\!\!\begin{array}{l}N(x)=
\frac{aq^2\left(1-q^{n}\right)\left(1-aq^{n}\right)\left(1-bq^{n}\right)\left(1-cq^{n}\right)
\left(c-abq^{n}\right)}{\left(1-abq^{2n}\right)^2\left(1-aq^{2n+1}\right)}-
\Big[\frac{q\big(-c+a(-1+(b+c)q^n)\big)}{1-abq^{2n}}+q^{\frac{1-n}{2}}x\Big]\times \\ [0.7em]
\hspace{2.8em}\left[\frac{c+a^2bq^n\big(-1-q+(1+b+c)q^{n+1}\big)+a\Big(
1-(b+c)\big(-1+q^n+q^{n+1}\big)-bcq^n\big(1+q-q^{n+1}\big)\Big)}{
q^{-n-1}\left(1-abq^{2n}\right)\left(1-abq^{2n+2}\right)}-x\right]
\end{array}
\]
\end{small}
and
\begin{small}
\[
\!\!\begin{array}{l}D(x)=
\frac{aq\left(1-q^{n}\right)\left(1-aq^{n}\right)\left(1-bq^{n}\right)\left(1-cq^{n}\right)
\left(c-abq^{n}\right)}{1-abq^{2n+1}}+\frac{1-q^{\frac{n}{2}}}{1-abq^{2n+2}}\times \\ [0.7em]
\left\{-c+a^2bq^{\frac{3n}{2}}\!\left(-1-q+(b\!+\!c)q^{n+1}\!-q^{1+\frac{n}{2}}\right)+
a\Big[-1+(b\!+\!c)\big(q^{\frac{n}{2}}+q^n+q^{n+1}\big)-\right. \\ [0.7em]
\left.bc\big(q^{\frac{3n}{2}}+q^{1+\frac{3n}{2}}+q^{2n+1}\big)\Big]
\Big[(c+a)q^{1+\frac{n}{2}}-a(b+c)q^{1+\frac{3n}{2}}-q^{\frac{1}{2}}\big(
1-abq^{2n}\big)x\Big]\right\},
\end{array}
\]
\end{small}
respectively.

\subsection{The second difference-recurrece relation}

If  we choose $\,\nu_1=n-1\,$,
$\,\nu_2=n\,$, $\,\nu_3=n+1\,$, $\,k_1=0\,$, $\,k_2=0\,$ and
$\,k_3=1\,$ in Theorem \ref{teo-rec-y}, and proceeding as in the previous case
one gets 
\begin{equation}\label{Eq4.13}
A_1(x)P_{n-1}(x;q)+A_2(x)P_{n}(x;q)+A_3(x)\Delta^{(1)}P_{n+1}(x;q)=0\,,
\end{equation}
where the coefficients $A_1$, $A_2$ and $A_3$, satisfy the linear relation
\begin{equation*}
\!\!\!\begin{array}{l}
A_3(x)\!\left[\left(q^{-\frac{n+1}{2}}-\!
\frac{B_{n+1}}{\alpha_{n+1}\tau_{n+1}^{\prime}B_{n+2}}\right)\!\big(x\!-\!\beta_{n+1}\big)+
\left(q^{-\frac{n+1}{2}}\beta_{n+1}+\frac{\tau_{n+1}(0)}{\tau_{n+1}^{\prime}}\right)\right]\!P_{n+1}+ \\ [1em]
\left[A_3(x)\frac{B_{n+1}}{\alpha_{n+1}\tau_{n+1}^{\prime}B_{n+2}}\gamma_{n+1}+
\left(\sigma(x)+\tau(x)\Delta x\left(s\!-\!\frac{1}{2}\right)\right)\frac{[n+1]_q}{\lambda_{n+1}}A_2(x)\right]P_{n}+ \\ [1em]
\left(\sigma(x)+\tau(x)\Delta x\left(s\!-\!\frac{1}{2}\right)\right)\frac{[n+1]_q}{\lambda_{n+1}}A_1(x)P_{n-1}
=0\,.
\end{array}
\end{equation*}
Comparing the above relation with the three-term recurrence relation \refe{Eq4.1} one
can obtain the explicit expressions for the coefficients $A_1$, $A_2$ and $A_3$ in
\refe{Eq4.13}.

\subsubsection{Some examples}

\subsubsection*{Al-Salam and Carlitz I polynomials}
Using the main data for the Al-Salam and Carlitz I polynomials we
obtain the relation
\begin{equation*}
A_1(x)U_{n-1}^{(a)}(x;q)+A_2(x)U_{n}^{(a)}(x;q)+A_3(x)\Delta^{(1)}U_{n+1}^{(a)}(x;q)=0
\end{equation*}
where
\begin{equation*}
\begin{split}
A_1(x)= & aq^{n-1}\left(1-q^{n}\right)x , \quad
A_2(x)=  \left[a\left(1+q^{\frac{n+1}{2}}\right)q^{n}-\Big((1+a)q^n-x\Big)x\right], \\
A_3(x)= & -a\frac{1-q}{1-q^{\frac{n+1}{2}}}q^{\frac{3n+1}{2}}.
\end{split}
\end{equation*}

\subsubsection*{Alternative $q$-Charlier polynomials}

In this case, one gets
\begin{equation*}
A_1(x)K_{n-1}(x;a;q)+A_2(x)K_{n}(x;a;q)+A_3(x)\Delta^{(1)}K_{n+1}(x;a;q)=0\,,
\end{equation*}
\begin{small}
\begin{equation*}
\begin{array}{l}
\!\!A_1(x)=\frac{a\big(1-q^n\big)\big(1+aq^{n-1}\big)
\left\{aq^n\big(1-q^{n+1}\big)+q^{-\frac{n+1}{2}}\big(1+aq^{2n+1}\big)
\Big[\big(1+aq^{n+1}\big)-q^{-\frac{n+1}{2}}\big(1+aq^{2n+2}\big)\Big]x\right\}}{
q^{2-3n}\big(1+aq^{2n-2}\big)\big(1+aq^{2n-1}\big)\big(1+aq^{2n}\big)} \,, \\ [1em]
\!\!A_2(x)=-x\left\{aq^n\big(1-q^{n+1}\big)+q^{-\frac{n+1}{2}}\big(1+aq^{2n+1}\big)
\Big[\big(1+aq^{n+1}\big)-q^{-\frac{n+1}{2}}\big(1+aq^{2n+2}\big)\Big]x\right\}+ \\ [0.5em]
\qquad\quad \frac{a^2q^{3n-1}\big(1\!-\!q^{n}\big)\big(1\!-\!q^{n+1}\big)+
q^{\frac{n-1}{2}}\big(1+aq^{n-1}+aq^{n}-aq^{2n}\big)\big(1+aq^{2n+1}\big)
\Big[\big(1+aq^{n+1}\big)-q^{-\frac{n+1}{2}}\!\big(1+aq^{2n+2}\big)\Big]x}{
\big(1+aq^{2n-1}\big)\big(1+aq^{2n+1}\big)}, \\ [1em]
\!\!A_3(x)=a(1-q)q^{\frac{n+1}{2}}\big(1+aq^{2n+1}\big)x^2
\end{array}
\end{equation*}
\end{small}

\subsection*{Concluding remarks}
In this paper we present a constructive approach for finding recurrence relations
for the hypergeo\-me\-tric-type functions on the linear-type
lattices, i.e., the solutions of the hypergeometric difference equation
\refe{difeq-q} on the linear-type lattices. Important instances of ``discret'' functions are
the celebrated Askey-Wilson polynomials and $q$-Racah polynomials.
Such functions are defined on the non-uniform lattice of the form
$x(s)= c_1(q) q^{s} + c_2(q) q^{-s} + c_2(q)$
with $c_1c_2\neq0$, i.e., a non linear-type lattice and therefore
they require a more detailed study (some preliminar general results can be found
in \cite{suslov}).

\subsection*{Acknowledgements}
The authors thank J. S. Dehesa and J.C. Petronilho for interesting discussions.
The authors were partially supported by DGES grants MTM2009-12740-C03;
PAI grant FQM-0262 (RAN) and Junta de Andaluc\'\i a grants P09-FQM-4643, Spain;
and CM-UTAD from UTAD (JLC).


\end{document}